\documentclass[12pt]{amsart}
\usepackage{xcolor}
\usepackage{mdwlist}
\usepackage{float}
\allowdisplaybreaks
\usepackage{longtable}
\usepackage{xcolor} % For color support
\usepackage{listings} % For code listings
\usepackage{amsmath} % For math environments
\usepackage{hyperref} % For clickable links

\definecolor{codegreen}{rgb}{0,0.6,0}
\definecolor{codegray}{rgb}{0.5,0.5,0.5}
\definecolor{codepurple}{rgb}{0.58,0,0.82}
\definecolor{backcolour}{rgb}{0.95,0.95,0.92}

\lstdefinestyle{sage}{
    backgroundcolor=\color{white},
    commentstyle=\color{codegreen},
    keywordstyle=\color{blue},
    numberstyle=\tiny\color{codegray},
    stringstyle=\color{codepurple},
    basicstyle=\footnotesize\ttfamily,
    breakatwhitespace=false,
    breaklines=true,
    captionpos=b,
    keepspaces=true,
    numbers=left,
    numbersep=5pt,
    showspaces=false,
    showstringspaces=false,
    showtabs=false,
    tabsize=2,
    frame=single,  % Add a box around the code
    rulecolor=\color{black}  % Color of the box
}

\lstdefinestyle{bw_sage}{
    language=Python,
    style=sage,
    basicstyle=\ttfamily\footnotesize,  % Small monospaced font (same as Sage style)
    keywordstyle=\color{black},         % Black color for keywords
    commentstyle=\color{black},         % Black color for comments
    stringstyle=\color{black},          % Black color for strings
    identifierstyle=\color{black},      % Black for variables/functions
    numberstyle=\color{black},          % Black for line numbers
    breaklines=true,                    % Allow line breaks
    frame=single,                        % Single box around the code
    captionpos=b                         % Caption position below
}

\usepackage{makecell}
\usepackage{ltablex}
\usepackage{amsfonts,amssymb,amsmath,textcomp}
\usepackage{enumitem}

\usepackage{tikz}
\usetikzlibrary{decorations.pathmorphing}
\usepackage{enumitem}

\usepackage[justification=centering]{caption}
\usepackage{supertabular}
 2   

\newcommand{\Q}{\mathbb{Q}}

\newcommand{\Z}{\mathbb{Z}}
\newcommand{\Po}{\operatorname{Po}}
\newcommand{\Cl}{\operatorname{Cl}}
\newcommand{\Gal}{\operatorname{Gal}}

\theoremstyle{plain}
\newtheorem{thm}{Theorem}[section]

\newtheorem{lem}[thm]{Lemma}
\newtheorem{prop}[thm]{Proposition}
\newtheorem{cor}[thm]{Corollary}
\newcommand{\thmref}[1]{Theorem~\ref{#1}}
\newcommand{\lemref}[1]{Lemma~\ref{#1}}
\newcommand{\propref}[1]{Proposition~\ref{#1}}
\newcommand{\corref}[1]{Corollary~\ref{#1}}
\theoremstyle{definition}
\newtheorem{defn}[thm]{Definition}

\numberwithin{equation}{section}
\begin{document}
\newcommand{\suchthat}{\;\ifnum\currentgrouptype=16 \middle\fi|\;}
\newcommand{\dmid}{\mathrel\Vert}
\title
[Rank of P\'olya Groups in Lecacheux Parametric Family]{Rank of P\'olya Groups in Lecacheux Parametric Family of Quintic Fields}
\author{Nimish Kumar Mahapatra}
\author{Prem Prakash Pandey}
\address[Nimish Kumar Mahapatra]{Indian Institute of Science Education and Research, Thiruvananthapuram, India.}
\email{nimish@iisertvm.ac.in}

\address[Prem Prakash Pandey]{Indian Institute of Science Education and Research, Berhampur, India.}
\email{premp@iiserbpr.ac.in}

\subjclass[2020]{11R29; 11R09}

\date{\today}

\keywords{P\'olya group; P\'olya field; Hilbert class field; Genus field}

\begin{abstract}
In this article, we study the P\'olya group of a new family of quintic fields, namely Lecacheux quintic fields. We show that the associated P\'olya groups can be arbitrarily large elementary abelian \(5\)-groups. Using density arguments, we prove that for every positive integer $k$, the set of odd integers $s$ such that the $5-$rank of the P\'olya group of the corresponding Lecacheux quintic field is at least $k$ has a positive density. Combining this with a result of Golod and Shafarevich, we see that for a positive proportion of $s$, the corresponding Lecacheux quintic fields admit an infinte $5-$class field tower. We also establish an upper bound for the P\'olya numbers of these fields in terms of the orders of their corresponding P\'olya groups. In addition, we prove that several fields in this family are non-monogenic despite having index one.

\end{abstract}

\maketitle{}

\section{Introduction}
Let \( K \) be an algebraic number field with ring of integers \( \mathcal{O}_K \), and let \( \text{Cl}(K) \) denote its class group. The polynomial ring \( K[X] \) plays a fundamental role in both algebraic geometry and algebraic number theory. A key object of study, in this context, is the set of polynomials in \( K[X] \) that map every element of \( \mathcal{O}_K \) back into \( \mathcal{O}_K \). This set forms a subring of \( K[X] \), known as the ring of integer-valued polynomials on \( \mathcal{O}_K \), and is denoted by \( \text{Int}(\mathcal{O}_K) \). It is defined as
\[ \text{Int}(\mathcal{O}_K) = \{ f \in K[X] \mid f(\mathcal{O}_K) \subseteq \mathcal{O}_K \}. \] 
It is evident that \( \mathcal{O}_K[X] \subseteq \text{Int}(\mathcal{O}_K) \subseteq K[X] \). Moreover, \( \text{Int}(\mathcal{O}_K) \) naturally admits a \( \mathcal{O}_K \)-module structure given by
\(
(\alpha, f(X)) \mapsto \alpha f(X),\text{ for } \alpha \in \mathcal{O}_K \text{ and } f(X) \in \text{Int}(\mathcal{O}_K)
\). The field $K$ is said to be a P\'olya field if $\text{Int}(\mathcal{O}_K)$ admits a regular basis. For each $n\in \mathbb{N}$, the leading coefficients of the polynomials in $\text{Int}(\mathcal{O}_K)$ of degree $n$ together with zero form a fractional ideal of $\mathcal{O}_K$ known as the characteristic ideal of index $n$ of $\mathcal{O}_K$, and it is denoted by $\mathfrak{J}_n(K)$ \cite[Proposition I.3.1]{PJC}. The characteristic ideals are given by 
\begin{equation*}
 \mathfrak{J}_n(K)=\{\text{leading coefficients of elements of } \text{Int}(\mathcal{O}_K) \text{ of degree } n\}\cup\{0\}.
\end{equation*}
For each integer $n\geq1$, let $[\mathfrak{J}_n(K)]$ be the ideal class in $\text{Cl}(K)$ corresponding to the fractional ideal $\mathfrak{J}_n(K)$. The P\'olya group $\Po(K)$ of $K$ is defined to be the subgroup of $\text{Cl}(K)$, generated by the classes $[\mathfrak{J}_n(K)]$.
A number field \( K \) is a P\'olya field if and only if \( \mathfrak{J}_n(K) \) is principal for all integers \( n \geq 1 \) \cite[Proposition II.1.4]{ZAN}. As $\text{Cl}(K)$ measures the failure of unique factorization in $\mathcal{O}_K$, the subgroup $\Po(K)$ measures the extent to which $K$ deviates from being a P\'olya field.

Classifying P\'olya fields of small degree has been an area of significant interest (see, for example, \cite{BAH1}, \cite{BAH2}, \cite{AL2}, \cite{ZAN}). Zantema \cite{ZAN} provided a complete characterization of quadratic P\'olya fields, while Leriche \cite{AL2} fully classified cyclic cubic P\'olya fields. In the same work, she also classified cyclic quartic and cyclic sextic P\'olya fields and extended similar classifications to certain families of bi-quadratic and sextic fields (see \cite[Theorem 5.1, Theorem 6.2]{AL2}). More recently, several attempts have been made to identify families of P\'olya (or non-P\'olya) fields in the remaining cases of bi-quadratic extensions (see \cite{BAH1}, \cite{BAH2}, \cite{AMA}, \cite{CHA}). In our recent work \cite{NKM}, we characterized the P\'olya-ness of a subfamily of Lehmer quintic fields and showed that these fields have large P\'olya groups.

In this article, we study P\'olya groups of a new family of quintic fields arising from Lecacheux parametric quintic polynomials. Particularly, we try to understand the distribution of the ranks of the P\'olya groups for this family. We believe that this article is the first instance where distribution of ranks of P\'olya groups in some family is studied. Odile Lecacheux \cite{LEC} introduced these parametric polynomials. For each \( s \in \mathbb{Q} \), the Lecacheux quintic polynomial \( f_s(x) \in \mathbb{Q}[X] \) is defined as follows:
\begin{equation}
f_s(X)=X^5+a_4(s)X^4+a_3(s)X^3+a_2(s)X^2+a_1(s)X+a_0(s),    
\end{equation}
where 
\begin{align*}
    a_4(s) &= s^5-3,\\
    a_3(s) &= -s^9-2s^8-3s^7-5s^6-6s^5-2s^4+s^3-s^2+3,\\
    a_2(s) &= s^{10}+2s^9+4s^8+6s^7+10s^6+9s^5+4s^4-2s^3+2s^2-1,\\
    a_1(s) &=-s^2(s^7+2s^6+3s^5+5s^4+5s^3+2s^2-s+1),\\
    a_0(s) &=s^5.
\end{align*}
In \cite{ALK} it is shown that if \( f_s(X) \) is irreducible in \( \mathbb{Q}[X] \), then the field \( K_s = \mathbb{Q}(\theta_s) \), where \( \theta_s \) is a root of \( f_s(X) \), is a cyclic quintic field. These fields are known as Lecacheux quintic fields. The main result proved in the article is the following:
\begin{thm}\label{NEW1}
For every positive integer $k\ge1$, the set of odd integers $s$ satisfying
\[
(\Z/5\Z)^k
\subseteq
\Po(K_s)
\]
has a positive density.
\end{thm}
Before proceeding further, we remark that it is a folklore conjecture that, for any two positive integers $m$ and $d$, the $m-$rank of class group of $K$, as $K$ runs through all number fields of degree $d$, is unbounded (see \cite{JEA}). The conjecture is out of reach, except in the cases when $m|d$ (in this case the conjecture follows from class field theory). Theorem \ref{NEW1} gives another justification for this conjecture for $m=d=5$. In fact, we obtain a stronger result, namely, for any positive integer $k$ there are infinitely many Lecacheux quintic fields with $5-$rank of their class group at least $k$. In \corref{cor1}, we obtain similar results for extension of degree $5t$ whenever $t$ is not divisible by $5$. Furthermore, combining \thmref{NEW1} with a result of Golod and Shafarevich \cite{ESG}, we obtain the following corollary.
\begin{cor}\label{CFT}
For a positive proportion of integers $s$, the $5-$class field tower of the Lecacheux quintic field $K_s$ is infinite.
\end{cor}
In fact, it is easy to give explicit Lecacheux quintic fields with infinite $5-$ class field towers. For example, for $s=-189,177$ and $183$, the corresponding Lecacheux quintic field has infinite $5-$class field tower (see Table \ref{tab1}).

Let \( K \) be a number field and \( \theta \in \mathcal{O}_K \) be a primitive element. The index \( [\mathcal{O}_K : \mathbb{Z}[\theta]] \) is known as the index of \( \theta \) in \( K \). It is denoted by \( I(\theta) \). The index of the number field \( K \) is defined as \( I(K) = \gcd\{I(\theta) \mid \theta \in \mathcal{O}_K \text{ and } K = \mathbb{Q}(\theta)\} \). If \( I(K) > 1 \), then \( K \) is not monogenic, meaning \( \mathcal{O}_K \neq \mathbb{Z}[\theta] \) for any \( \theta \in K \). However, the converse is not true in general, as there exist non-monogenic number fields \( K \) with \( I(K) = 1 \). These fields are not monogenic, but this non-monogeneity is not due to local reasons. For further details, we refer the reader to \cite{IND}. Our next theorem produces several examples of such fields.
\begin{thm}\label{NEW3}
Let \( s \neq \pm1 \) be an odd integer, and let \( E_1(s,1) \) and \( E_2(s,1) \) be cube-free. Then, the Lecacheux quintic field \( K_s \) is not monogenic. Moreover, the field index \( I(K_s) = 1 \) whenever \( s \neq \pm1 \) is an odd integer such that \( 3 \nmid s \).
\end{thm}
In Section 2, we develop the preliminaries necessary to prove \thmref{NEW1}. Section 3 is dedicated to the proof of \thmref{NEW1}, and we also investigate large P\'olya groups arising in certain suitable compositum fields. In Section 4, we investigate the P\'olya numbers and monogeneity, along with the proof of \thmref{NEW3}. Finally, in Section 5, we present some computations performed by us using SageMath and Wolfram Mathematica software.

\section{Preliminaries}
In this section we assume that the number field $K$ is a Galois extension of $\mathbb{Q}$. For any prime number $p$, the ramification index of $p$ in $K/\mathbb{Q}$ is denoted by $e_p$.  In \cite{JLC}, Chabert obtained a nice description for the cardinality of $\Po(K)$ for cyclic extensions $K/\mathbb{Q}$.
\begin{prop}\label{CHABERT}
\cite[Corollary 3.11]{JLC} Assume that the extension $K/\mathbb{Q}$ is cyclic of degree $n$.
\begin{enumerate}
    \item If $K$ is real and $N(\mathcal{O}_K^{\times})=\{1\}$, then $|\Po(K)|=\frac{1}{2n}\times\prod_p e_p$.
    \item In all other cases, $|\Po(K)|=\frac{1}{n}\times\prod_p e_p$.
\end{enumerate}
\end{prop}
When $K$ is a cyclic number field of odd degree, all ramification indices $e_p$ are odd, and case (1) of Proposition \ref{CHABERT} does not occur. We record this in the following corollary.
\begin{cor}\label{cor2}
If $K/\mathbb{Q}$ is cyclic extension of odd degree $n$ and $n$, then $|\Po(K)|=\frac{1}{n}\times\prod_p e_p$.
\end{cor}
 An ideal $\mathfrak{a}$ of $K$ is called an ambiguous ideal of $K$ if it is invariant under the action of $\text{Gal}(K/\mathbb{Q})$, i.e., $\mathfrak{a}^\tau=\mathfrak{a} \text{ for all }\tau\in \text{Gal}(K/\mathbb{Q})$. Zantema \cite[\S 3]{ZAN} showed that $\Po(K)$ is the subgroup of $\text{Cl}(K)$ generated by the classes of the ambiguous ideals of $K$.  In other words, 
\begin{equation}\label{2ambi}
    \Po(K)=\{[\mathfrak{a}]\in \text{Cl}(K) : \mathfrak{a}^\tau=\mathfrak{a}\text{ for all }\tau\in \text{Gal}(K/\mathbb{Q})\}.
\end{equation}
It is easy to see that for odd integers $s$, the Lecacheux quintic polynomials are irreducible.
\begin{lem}\label{lem1}
   If $s$ is an odd integer, then $f_s(X)$ is irreducible in  $\mathbb{Q}[X]$.
\end{lem}
\begin{proof}
    Reducing the coefficients of $f_s(X)$ modulo 2 we get
    $$f_s(X)\equiv X^5+X^2+1 \pmod 2.$$
    A direct check shows that the reduced polynomial $X^5+X^2+1$ is irreducible of degree 5 in  $\mathbb{F}_2[X]$.
\end{proof}

For any pair of integers $u, v$ we use the following notations:
\begin{equation*}
    E_1=E_1(u,v)= u^4-2u^3v+4u^2v^2-3uv^3+v^4
\end{equation*}
and
\begin{equation*}
    E_2=E_2(u,v)= u^4+3u^3v+4u^2v^2+2uv^3+v^4.
\end{equation*}
When there is no confusion, we write $E_1$ and $E_2$ respectively for $E_1(u,v)$ and $E_2(u,v)$. 

Now we recall the following result, which explicitly describes the conductors of the Lecacheux quintic fields.
\begin{thm} \label{cond}\cite[Theorem 3]{ALK}
 Let $s=u/v$ be a rational number and assume that $f_s(X)$ is irreducible in $\mathbb{Q}[X]$. Then the conductor $f(K_s)$ is given by
 $$f(K_s)= 5^{\alpha}\prod\limits_{\substack{q\equiv 1\pmod 5 \\ q\mid E_1E_2 \\ v_q(E_1E_2)\not\equiv 0\pmod 5}}q,$$
 where $q$ runs through primes, and
 $$\alpha=\begin{cases}
     0, \text{ if } 2u-v\not\equiv 0 \pmod5,\\
     2, \text{ if } 2u-v \equiv 0 \pmod 5.
 \end{cases}$$
 Here $v_q(E_1E_2)$ is the largest integer such that $q^{v_q(E_1E_2)}\mid E_1E_2$. 
\end{thm}

The following lemma records the precise condition under which the conductor of the Lecacheux quintic field is divisible by $5^2$.

\begin{lem}\label{newlem}
Let $s$ be an odd integer, and let $K_s$ be the corresponding Lecacheux quintic field. Let $$E_1(s,1) = s^4 - 2s^3 + 4s^2 - 3s + 1$$ and $$E_2(s,1) = s^4 + 3s^3 + 4s^2 + 2s + 1.$$ Then,
$
5^2 \mid f(K_s)$  if and only if  $5 \mid E_1(s,1)$  and $5 \mid E_2(s,1)$.
\end{lem}

\begin{proof}
By Theorem \ref{cond}, the conductor of $K_s$ is given by
$$
f(K_s) = 5^{\alpha} \prod_{\substack{q \equiv 1 \pmod{5} \\ q \mid E_1E_2\\ v_q(E_1E_2)\not\equiv 0\pmod 5}} q,
$$
where 
$$
\alpha = 
\begin{cases} 
0, & \text{if } 2s - 1 \not\equiv 0 \pmod{5}, \\
2, & \text{if } 2s - 1 \equiv 0 \pmod{5}.
\end{cases}
$$
It is clear from the above formula that, $5^2 \mid f(K_s)$ if and only if $\alpha = 2$, which is equivalent to
$$
2s - 1 \equiv 0 \pmod{5} \iff s \equiv 3 \pmod{5}.
$$
Note that 
$$E_1(s,1) \equiv E_2(s,1) \pmod 5.$$
A direct check of the residue classes modulo $5$ yields
$$
 E_2(s,1) \equiv 
\begin{cases} 
0 \pmod{5}, & \text{if } s \equiv 3 \pmod{5}, \\
1 \pmod{5}, & \text{otherwise}.
\end{cases}
$$
Combining these, we see that $5 \mid E_1(s,1)$ and $5 \mid E_2(s,1)$ if and only if $s \equiv 3 \pmod{5}$, which is precisely the condition for $\alpha = 2$. This completes the proof.
\end{proof}

Next, we recall a landmark result of Hooley concerning the distribution of power-free values of irreducible polynomials.
\begin{thm}\cite{Hooley}\label{Hooley}
Let \(f(X)\in \mathbb{Z}[X]\) be an irreducible polynomial of degree \(r\ge 3\), and assume that there is no fixed \((r-1)\)-th power, greater than $1$, dividing each $f(n)$. Let
\[
N(x)=\#\{n\le x : f(n)\ \text{is }(r-1)\text{-power-free}\}
\text{ and }
\]
\[
\rho_f(m)=\#\{a \bmod m : m\mid f(a)\}, \text{ for each positive integer }m.
\]
Then
\begin{equation}\label{Prem1}
N(x)=c_f x+O\!\left(\frac{x}{(\log x)^{A/\log \log \log x}}\right),
\end{equation}
where

\[
c_f=\prod_{q}\left(1-\frac{\rho_f(q^{\,r-1})}{q^{\,r-1}}\right),
\]
the product runs on all primes \(q\). In the equation $(\ref{Prem1})$ \(A>0\) is a constant that depends only on \(f\).
\end{thm}
The following result on the normal order of $\omega(f(n))$ follows from the work of Cojocaru and Murty \cite{RAM}.

\begin{thm}\cite[Theorem 3.2.3]{RAM}\label{dist}
Let $f(X)\in\mathbb Z[X]$ be a non‑constant irreducible polynomial. Then for every fixed $\varepsilon>0$,
\[
\#\Bigl\{n\le x : \bigl|\omega(f(n))-\log\log n\bigr| > \varepsilon\log\log n\Bigr\} = o(x) \qquad (x\to\infty).
\]
In other words, $\log\log n$ is the \emph{normal order} of $\omega(f(n))$.
\end{thm}

Lastly, we recall a result on the number of integral solutions of the Diophantine equation of the type $Y^m = f(X)$. Particularly when $m = 2$ and $f(x)$ is a monic quartic polynomial then the following result due to Masser \cite{MAS}, gives us the specific bound on the integral solutions to the curve.
\begin{thm}\label{mass}\cite{MAS}
Consider the Diophantine equation $Y^2=f(X)$,
where $f(X)$ is a polynomial of degree four with integer coefficients. Assume that $f(X)$ is monic and its discriminant is not a perfect square. Then any integer solution $(x, y)$ of the above equation satisfies
$$|x| \leq 26 H(f)^3,$$
where $H(f)$ denotes the maximum of the absolute values of the coefficients of $f(X)$.
\end{thm}

\section{Proof of \thmref{NEW1}}
Before proving \thmref{NEW1}, we mention that there are computer algebra scripts to find elements of $\Q(\zeta_5)$ with a given absolute norm. The details of one such Mathematica script are given in the last section. Using this, we see that for any odd integer $s$,
\begin{equation}\label{2norm}
    N_{\Q(\zeta_5)/\Q}(s+\zeta_5+\zeta_5^3)=E_1(s,1)
\end{equation}
and
\begin{equation}\label{n2}
  N_{\mathbb{Q}(\zeta_5)/\mathbb{Q}}(s+1+\zeta_5)=E_2(s,1) .
\end{equation}
The following propositions are crucial for the proof of \thmref{NEW1}.
\begin{prop}\label{prop1}
    Let $s$ be any odd integer and $p$ be any prime different from $5$, dividing $E_1(s,1)$. Then the congruence $p\equiv 1 \pmod5$ holds.
\end{prop}
\begin{proof}We prove this result by showing that any prime divisor $p$, different from 5, of $E_1(s,1)$ splits completely in $\mathbb{Q}(\zeta_5)/\mathbb{Q}$ for any odd integer $s$. Suppose $p$ does not split completely in $\mathbb{Q}(\zeta_5)/\mathbb{Q}$. Therefore the number of prime factors of $p$ in $\mathbb{Q}(\zeta_5)$ is at most 2. 

Let us consider the case if $p\mathbb{Z}[\zeta_5]$ is a prime ideal in $\mathbb{Z}[\zeta_5]$. Since $p$ divides $ E_1(s,1)$, from Eq.(\ref{2norm}), we see that $p\mathbb{Z}[\zeta_5]$ divides $ (s+\zeta_5+\zeta_5^3)^{\sigma}$ for some $\sigma\in \text{Gal}(\mathbb{Q}(\zeta_5)/\mathbb{Q})$. Without loss of generality, we can write
\begin{equation}
    s+\zeta_5+\zeta_5^3 \equiv 0\pmod{p \mathbb{Z}[\zeta_5]}.
\end{equation}
Since $\{1, \zeta_5, \zeta_5^2, \zeta_5^3\}$ forms a $\mathbb{Z}$-basis of $\mathbb{Z}[\zeta_5]$, the above congruence implies that each coefficient on the left-hand side must be divisible by $p$. However, this implies that $p$ divides 1, which is a contradiction.

Now consider the case if $p\mathbb{Z}[\zeta_5]=\mathfrak{p}_1\mathfrak{p}_2$. This implies $\mathfrak{p}_1$ divides $ N_{\mathbb{Q}(\zeta_5)/\mathbb{Q}}(s+\zeta_5+\zeta_5^3)$. Therefore $\mathfrak{p}_1$ contains one of the factors of $ N_{\mathbb{Q}(\zeta_5)/\mathbb{Q}}(s+\zeta_5+\zeta_5^3)$. Thus, we can write 
\begin{equation}\label{eqa}
    s+\zeta_5+\zeta_5^3 \equiv0\pmod { (\mathfrak{p}_1)^{\sigma}},\text{ for some }\sigma\in \text{Gal}(\mathbb{Q}(\zeta_5)/\mathbb{Q}).
\end{equation}
For $\sigma\in \text{Gal}(\mathbb{Q}(\zeta_5)/\mathbb{Q})$, $ (\mathfrak{p}_1)^{\sigma}=\mathfrak{p}_2$ or $ (\mathfrak{p}_1)^{\sigma}=\mathfrak{p}_1$. Without loss of generality, assuming $ (\mathfrak{p}_1)^{\sigma}=\mathfrak{p}_2$ Eq.(\ref{eqa}) reduces to 
\begin{equation}\label{eqaa}
    s+\zeta_5+\zeta_5^3 \equiv0\pmod {\mathfrak{p}_2}.
\end{equation}
Let $\tau\in \text{Gal}(\mathbb{Q}(\zeta_5)/\mathbb{Q})$ such that $\tau : \zeta_5 \rightarrow \zeta_5^2$. Since $\tau$ is an element of order 4 in $\text{Gal}(\mathbb{Q}(\zeta_5)/\mathbb{Q})$ and the decomposition groups of $\mathfrak{p}_1$ and $\mathfrak{p}_2$ are of order 2, $\tau$ can not fix $\mathfrak{p}_1$ or $\mathfrak{p}_2$.
Applying $\tau$ to Eq.(\ref{eqaa}) we get
\begin{equation}\label{eqb}
   s+\zeta_5+\zeta_5^2\equiv 0\pmod{({\mathfrak{p}_2})^{\tau}}.    
\end{equation}
Multiplying Eq.(\ref{eqaa}) and Eq.(\ref{eqb}) we get
\begin{equation}
  s^2+(2s-1)\zeta_5+s\zeta_5^2+s\zeta_5^3  \equiv 0\pmod{p \mathbb{Z}[\zeta_5]}. 
\end{equation}
Since \( \{ 1, \zeta_5, \zeta_5^2, \zeta_5^3 \} \) is a \( \mathbb{Z} \)-basis of \( \mathbb{Z}[\zeta_5] \), \( p \) divides \( s \) and \( 2s - 1 \). Thus, \( p \) divides \( 1 \), which is a contradiction.
\end{proof}
Using Eq.(\ref{n2}), along similar lines, one can obtain the following proposition.
\begin{prop}\label{prop2}
    Let $s$ be any odd integer and $p$ be any prime different from 5, dividing $E_2(s,1)$. Then, the congruence, $p\equiv 1 \pmod5$, holds.
\end{prop}
The following result describes the P\'olya group of Lecacheux quintic fields and plays an important role in the proof of Theorem \ref{NEW1}.

\begin{thm} \label{NEW1.1}
Let $s$ be any odd integer, and $K_s$ denote the Lecacheux quintic field. Let $E_1=E_1(s,1)=s^4-2s^3+4s^2-3s+1$ and $E_2=E_2(s,1)=s^4+3s^3+4s^2+2s+1$. If $E_1$ and $E_2$ are cube-free, then
\begin{enumerate}[label=(\roman*)]
    \item $K_s$ is a P\'olya field if and only if $s=\pm1$,
    \item
    $\Po(K_s)\cong 
\left(\frac{\mathbb{Z}}{5\mathbb{Z}}\right)^{\omega({E_1E_2})-1}$,
    where $\omega(t)$ denotes the number of distinct prime divisors of $t$.
\end{enumerate}
\end{thm}

\begin{proof} Since \( E_1 \) and \( E_2 \) are assumed to be cube-free, using \lemref{lem1} and \thmref{cond} we find that the conductor of $K_s$ is
$$f(K_s)= 5^{\alpha}\prod\limits_{\substack{q\equiv 1\pmod 5 \\ q\mid E_1E_2} } q,$$
 where $q$ runs through primes, and
 $$\alpha=\begin{cases}
     0, \text{ if } 2s-1\not\equiv 0 \pmod5,\\
     2, \text{ if } 2s-1 \equiv 0 \pmod 5.
 \end{cases}$$

Thus, using \lemref{newlem}, \propref{prop1} and \propref{prop2}, we see that the primes ramifying in the extension $K_s/\Q$ are precisely the prime divisors of $E_1E_2$. Further, the ramification index of each such prime is 5. 

From \corref{cor2} it follows that $|\Po(K_s)|=5^{\omega(E_1E_2)-1}$. Therefore, $K_s$ is a P\'olya field if and only if $f(K_s)$ is a prime power. If $s=\pm1$, then $E_1E_2=11$. From \corref{cor2} we get $|\Po(K_{\pm1})|=1$. This implies $K_1$ and $K_{-1}$ are P\'olya fields. Now suppose that $s\neq\pm1$. We claim that $E_1E_2$ can not be a prime power. Suppose $E_1E_2$ is a prime power then $E_1=p^a$ and $E_2=p^b$ for some prime $p$ and positive integers $a$ and $b$. Since $E_1$ and $E_2$ are cube-free this implies $a,b<3$. Now consider the curves
\begin{align}\label{curve1}
 Y^2&=f(X)=X^4-2X^3+4X^2-3X+1\\\label{curve2}
 Y^2&=g(X)=X^4+3X^3+4X^2+2X+1.   
\end{align}
From \thmref{mass}, we see that any integral solution $(x,y)$ of (\ref{curve1}) and (\ref{curve2}) satisfies 
$$|x|\leq 26\times4^3=1664.$$
Using a SageMath program we find that for $x\in[-1664,1664]$  the only integral solutions to the curve $Y^2=f(X)$ are $(x,y) = (0, 1)$ and $(1,1)$ and the only solutions to the curve $Y^2=g(X)$ are $(x,y) = (-4, 11)$, $(-1,1)$ and $(0,1)$. In particular, it is not possible that  $E_1=p^2$ or $E_2=p^2$ for some odd integer $s\neq\pm1$. Thus  $a,b<2$. Next we observe that $E_1<E_2$ whenever $s>0$, and $E_1>E_2$ whenever $s<0$. Thus, $E_1\neq E_2$ for all odd integers $s$ and hence $a=b=1$ is not possible. This contradicts the fact that $a$ and $b$ are positive integers. This establishes the claim. Consequently $5\mid |\Po(K_s)|$. This proves $(i)$ of \thmref{NEW1}.
 
We have $\text{Gal}(K_s/\mathbb{Q})\simeq \mathbb{Z}/5\mathbb{Z}$. Let $\sigma$ be a generator of $\text{Gal}(K_s/\mathbb{Q})$ and $[\mathfrak{J}]\neq [1]$ be an ambiguous ideal class of $K_s$. Then we have,
\begin{align*}
    [\mathfrak{J}]^5&=[\mathfrak{J}][\mathfrak{J}][\mathfrak{J}][\mathfrak{J}][\mathfrak{J}]\\    
    &=[\mathfrak{J}][\mathfrak{J}]^\sigma[\mathfrak{J}]^{\sigma^2}[\mathfrak{J}]^{\sigma^3}[\mathfrak{J}]^{\sigma^4}\\
    &=[\beta \mathcal{O}_{K_s}]\\
    &=[(1)]
\end{align*}
where $|\beta|=N(\mathfrak{J})=[\mathcal{O}_{K_s}:\mathfrak{J}]$ denotes the norm of the ideal $\mathfrak{J}$ and $[(1)]$ denotes the trivial ideal class of $\text{Cl}({K_s})$. We conclude that the order of any non-trivial ambiguous ideal class in the class group of $K_s$ is $5$. Using (\ref{2ambi}) and structure theorem for finite abelian groups we obtain
$$\Po(K_s)\cong
\left(\frac{\mathbb{Z}}{5\mathbb{Z}}\right)^{\omega({E_1E_2})-1}.$$
This finishes the proof of $(ii)$ of \thmref{NEW1}.
\end{proof}
The following theorem is another crucial ingredient for the proof of \thmref{NEW1}.

\begin{thm}\label{simul}
Let $f(X)=X^{4}-2X^{3}+4X^{2}-3X+1$ and $g(X)=X^{4}+3X^{3}+4X^{2}+2X+1.$ Then the lower natural density of the set 
\[
\mathcal C=\{n\in\mathbb{Z}_{>0} : n \text{ is odd and  both } f(n)\ \text{and}\ g(n)\ \text{are cube-free}\}
\]
 is positive. In particular, there exist infinitely many odd integers $n$ for which both $f(n)$ and $g(n)$ are cube-free.
\end{thm}

\begin{proof}
Consider the following sets
\[
\mathcal A_f=\{n\in\mathbb{Z}_{>0} : f(n)\ \text{is cube-free}\},\qquad
\mathcal A_g=\{n\in\mathbb{Z}_{>0} : g(n)\ \text{is cube-free}\}.
\]
By \thmref{Hooley}, since $f$ and $g$ are irreducible of degree $4$ and have no fixed prime divisor, both $A_f$ and $A_g$ have natural densities
\[
c_f=\prod_{q}\left(1-\frac{\rho_f(q^3)}{q^3}\right),\qquad
c_g=\prod_{q}\left(1-\frac{\rho_g(q^3)}{q^3}\right),
\]
where
$\rho_h(m)=\#\{a \bmod m :\ m \mid h(a)\}$.
We now bound these densities from below. Since
\[
\operatorname{disc}(f)=\operatorname{disc}(g)=125,
\]
for every prime $q\neq5$ the reduction of each polynomial modulo $q$ is separable. Hence, every root modulo $q$ is simple and lifts uniquely to a root modulo $q^3$ by Hensel’s lemma and therefore
\[
\rho_h(q^3)\le 4 \qquad (q\neq5,\ h=f,g).
\]
For the primes $q\le 5$, a direct calculation gives
\[
\rho_f(q^3)=\rho_g(q^3)=0.
\]
Consequently, for $h=f,g$,
\[
c_h \ge \prod_{q>5}\left(1-\frac{4}{q^3}\right).
\]
Let
\[
t_q=\frac{4}{q^3} \qquad (q>5).
\]
Then \(0<t_q<1/2\), so
\[
\log(1-t_q)\ge -t_q-t_q^2.
\]
Hence
\[
\prod_{q>5}\left(1-\frac{4}{q^3}\right)
\ge
\exp\left(
-\sum_{q>5}\frac{4}{q^3}
-\sum_{q>5}\frac{16}{q^6}
\right).
\]
Now
\[
\sum_{q>5}\frac{4}{q^3}
\le
4\sum_{n=7}^{\infty}\frac1{n^3}
\le
4\int_6^\infty \frac{dx}{x^3}
=
\frac1{18},
\]
and
\[
\sum_{q>5}\frac{16}{q^6}
\le
16\sum_{n=7}^{\infty}\frac1{n^6}
\le
16\int_6^\infty \frac{dx}{x^6}
=
\frac{16}{5\cdot 6^5}
<0.0007.
\]
Therefore
\[
c_h
\ge
\exp\left(-\frac1{18}-0.0007\right)
>0.945.
\]
In particular,
\[
c_f+c_g>1.
\]
Consider the set
\[
\mathcal A=\{n\in\mathbb{Z}_{>0} : f(n)\ \text{and}\ g(n)\ \text{are both cube-free}\}.
\]
Then, for every $x$,
$|\mathcal A\cap[1,x]|\ge |\mathcal A_f\cap[1,x]|+|\mathcal A_g\cap[1,x]|-x$,
and hence
\[
\liminf_{x\to\infty} \frac{|\mathcal A\cap[1,x]|}{x}\ge c_f+c_g-1>\frac{1}{2}.
\]
Thus, the lower natural density of $\mathcal A$ is positive, and is greater than $1/2$. Consequently,
%Now, since the odd integers have natural density $1/2$ and $\mathcal A$ has lower density exceeding $1/2$, the set of odd integers in $\mathcal A$ must have positive lower density; indeed, if 
%\[
%\mathcal C=\{n\in\mathbb{Z}_{>0} : n \text{ is odd and  %both } f(n)\ \text{and}\ g(n)\ \text{are cube-free}\},
%\]
%then
\[
\liminf_{x\to\infty} \frac{|\mathcal C\cap[1,x]|}{x}
\ge \liminf_{x\to\infty} \frac{|\mathcal A\cap[1,x]|}{x} - \frac12 > 0.
\]
Therefore the odd integers $n$ with $f(n)$ and $g(n)$ both cube-free form a set of positive lower density, which implies the existence of infinitely many such odd integers.
\end{proof}

\begin{proof}[Proof of $\thmref{NEW1}$] Let $\mathcal{C}$ be as in \thmref{simul}. Then the lower natural density of the set $\mathcal C$ is positive. That is, there exists $\delta>0$ such that
\[
\liminf_{x\to\infty} \frac{|\mathcal C\cap[1,x]|}{x} = \delta > 0.
\]
Fix $k \geq 1$. Now, applying \thmref{dist} to $E_1$ and $E_2$, we obtain that both of the following sets 
    \[
    \mathcal B_k^{(1)} = \{n\in\mathbb{Z}_{>0} : \omega(E_1(n)) \ge k+1 \,\}\quad
   \text{ and } \quad
    \mathcal B_k^{(2)} = \{n\in\mathbb{Z}_{>0} : \omega(E_2(n)) \ge 1 \,\}
    \]
    have natural density $1$. Consequently the intersection $\mathcal B_k := \mathcal B_k^{(1)} \cap \mathcal B_k^{(2)}$ also has density $1$. For every sufficiently large integer $n\in B_k$, we have
\[
\omega(E_1(n))\ge k+1
\qquad\text{and}\qquad
\omega(E_2(n))\ge1.
\]
Therefore,
\[
\omega(E_1(n)E_2(n))
\ge
\omega(E_1(n))+\omega(E_2(n))-1
\ge k+1.
\]
Now consider $\mathcal C \cap \mathcal B_k$. Since $\mathcal B_k$ has density $1$, its complement has density $0$. Therefore
\[
\liminf_{x\to\infty} \frac{|\mathcal C\cap\mathcal B_k\cap[1,x]|}{x}
\ge \liminf_{x\to\infty} \frac{|\mathcal C\cap[1,x]|}{x} - \limsup_{x\to\infty} \frac{|\mathcal B_k^c\cap[1,x]|}{x} = \delta - 0 = \delta > 0.
\]
Thus $\mathcal C\cap\mathcal B_k$ has positive density; in particular, it is infinite. Take any $s\in\mathcal C\cap\mathcal B_k$. Then by Theorem~\ref{NEW1.1} we obtain
$
(\Z/5\Z)^k \subseteq \Po(K_s).
$
This completes the proof.
    
\end{proof}

As an application of \thmref{NEW1}, we obtain the following corollary.

\begin{cor}\label{cor1}
Let \(d=5t\) with \(5\nmid t\). Then there exist infinitely many Galois extensions \(L/\mathbb{Q}\) of degree \(d\) whose P\'olya groups have arbitrarily large \(5\)-rank.
\end{cor}
\begin{proof}
First, we construct an auxiliary P\'olya field $F$ of degree $t$. 
Write
\[
t=\prod_{i=1}^r p_i^{e_i},
\]
where the $p_i$ are distinct primes and $p_i\neq 5$ for all $i$.
For each $i$, choose a prime $q_i$ such that
\[
q_i\equiv 1 \pmod{p_i^{e_i}}.
\]
%Such primes exist by Dirichlet's theorem on arithmetic progressions, and there are infinitely many choices for each $i$. 
Let $F_i$ be the unique subfield of $\Q(\zeta_{q_i})$ of degree $p_i^{e_i}$. The only prime ramified in $F_i$ is $q_i$, and it is totally ramified. Since
\[
\Gal(\Q(\zeta_{q_i})/\Q)\cong (\Z/q_i\Z)^\times
\]
is cyclic of order $q_i-1$, the field $F_i/\Q$ is cyclic of degree $p_i^{e_i}$. By \propref{CHABERT}, each $F_i$ is a P\'olya field.
Now consider
\[
F=F_1F_2\cdots F_r.
\]
Since the conductors $q_1,\dots,q_t$ are pairwise distinct prime numbers, the fields $F_i$ are disjoint over $\Q$. Moreover, by \cite[Theorem~3.7]{ChabertPolya}, the compositum of linearly disjoint P\'olya fields is again a P\'olya field. Hence $F$ is a P\'olya field. Hence
\[
[F:\Q]=\prod_{i=1}^r [F_i:\Q]
=\prod_{i=1}^r p_i^{e_i}
=t.
\]
Finally, since there are infinitely many admissible choices for the primes \(q_i\), this construction yields infinitely many pairwise distinct fields \(F\). Fix one such field and denote it by \(F\).

Fix $k\ge 1$. By Theorem~\ref{NEW1}, there are infinitely many odd integers $s$ such that $(\Z/5\Z)^k \subseteq \Po(K_s)$. Now consider the field
\[
L=K_sF.
\]
Note that
\[
K_s\cap F=\Q \quad\text{ and, hence,  }\quad [L:\Q]=[K_s:\Q]\,[F:\Q]=5t.
\]
Since $K_s$ and $F$ are linearly disjoint, again by \cite[Theorem~ 3.7]{ChabertPolya} we have
\[
\Po(L)\cong \Po(K_s)\oplus \Po(F).
\]
Because $F$ is a P\'olya field, $\Po(F)$ is trivial, and therefore
\[
\Po(L)\cong \Po(K_s).
\]
In particular,
$(\Z/5\Z)^k \subseteq \Po(K_s)\cong \Po(L)$, so $L$ has large P\'olya group.
\end{proof}

The following fundamental theorem of Golod and Shafarevich plays a central role in the study of class field towers and will be needed in the proof of \corref{CFT}.

\begin{thm}\label{InfiniteTower}\cite{ESG}
Let $K$ be a number field, and let $\ell$ be a prime number. If
\[
\operatorname{rk}_{\ell}\bigl(\Cl(K)\bigr)
>
2+2\sqrt{\operatorname{rk}\!\left(\mathcal{O}_K^{\times}\right)+1},
\]
where $\operatorname{rk}_{\ell}(\Cl(K))$ denotes the $\ell$-rank of the class group $\Cl(K)$.Then the $\ell$-class field tower of $K$ is infinite.
\end{thm}
\begin{proof}[Proof of $\corref{CFT}$]
Since $\operatorname{Po}(K_s)$ embeds into $\Cl(K_s)$, we have
\[
\operatorname{rk}_5(\Cl(K_s)) \ge \operatorname{rk}_5(\operatorname{Po}(K_s)).
\]
By \thmref{NEW1}, for $k=7$, the set of odd integers $s$ satisfying
\[
\operatorname{rk}_5(\operatorname{Po}(K_s)) \ge 7
\]
has positive density. For each such $s$, we get $\operatorname{rk}_5(\Cl(K_s)) \ge 7$. Since $[K_s:\mathbb{Q}]=5$, the unit rank satisfies
\[
\operatorname{rk}(\mathcal{O}_{K_s}^{\times}) \le 4.
\]
Therefore,
\[
2+2\sqrt{\operatorname{rk}(\mathcal{O}_{K_s}^{\times})+1}
\le 2+2\sqrt{5} < 7.
\]
Thus, the condition
\[
\operatorname{rk}_5(\Cl(K_s)) > 2+2\sqrt{\operatorname{rk}(\mathcal{O}_{K_s}^{\times})+1}
\]
is satisfied, hence, by \thmref{InfiniteTower}, the $5$-class field tower of $K_s$ is infinite. This completes the proof.
\end{proof}

\section{P\'olya numbers and monogeneity  of Lecacheux quintic fields}

To prove \thmref{NEW3} we need the following result of Gras \cite{MNG} on the monogeneity  of cyclic number fields of prime degree.
\begin{prop}\label{gras}
\cite[\S 5, Théorème]{MNG} If $K$ is a cyclic number field of prime degree $p\geq5$ then $K$ is monogenic only if $2p+1$ is prime and $K$ is the maximal real subfield of the cyclotomic field $\mathbb{Q}(\zeta_{2p+1})$.
\end{prop}
We also need a result of Zylinski.
\begin{prop}\label{zyl} \cite{ZAL}
If $K$ is a number field of degree $n$, then $I(K)$ has only prime divisors $p$ satisfying $p<n$.
\end{prop}

\begin{proof} [Proof of $\thmref{NEW3}$]
Given that \( K_s \) is a Lecacheux quintic field such that \( s \neq \pm1 \), and \( E_1(s,1) \) and \( E_2(s,1) \) are cube-free. It follows from Theorem \ref{NEW1.1} that
\begin{equation}
    |\Po(K_s)|\geq5.
\end{equation}
On the other hand, it is known that both cyclotomic and real cyclotomic fields are P\'olya fields \cite[Proposition 2.6]{ZAN}. Therefore, \( K_s \) cannot be the maximal real subfield of a cyclotomic field. Furthermore, by Proposition \ref{gras}, it follows that \( K_s \) is not monogenic.

Next, we show that $I(K_s)=1$, for any odd integer $s\neq\pm1$ such that $3\nmid s$.  We recall the relation
\begin{equation}
d(\theta_s)=[I(\theta_s)]^2d_{K_s},
\end{equation}
where $d(\theta_s)$ and $d_{K_s}$ denote the discriminant of the polynomial $f_s$ and the field discriminant of $K_s$, respectively. Using Wolfram Mathematica software, we obtain that
\begin{align*}
d(\theta_s)&=s^{18}(s^4-2s^3+4s^2-3s+1)^4(s^4+3s^3+4s^2+2s+1)^8\\
&=s^{18}(E_1(s,1))^4(E_2(s,1))^8.    
\end{align*}
From \propref{prop1} and \propref{prop2}, we see that neither of $E_1$ or $E_2$ is divisible by $2$ or $3$. Consequently, since $s$ is assumed to be odd and $3\nmid s$ we conclude that $I(\theta_s)$ is not divisible by $2$ or $3$. Now, from \propref{zyl}, it follows that $I(K_s)=1$.
\end{proof}

In 1964, Golod and Shafarevich \cite{ESG} gave a negative answer to the classical embedding problem: Is every number field $K$ contained in a field $L$ with class number one? Leriche \cite{AL3} investigated the corresponding problem for P\'olya fields set up and proved that every number field is contained in a P\'olya field, namely its Hilbert class field.

\begin{defn} \cite{AL3}
The P\'olya number of $K$ is given by the integer 
$$po_K=\text{min}\{[L : K]\mid K\subseteq L,\: L\text{ P\'olya field}\}.$$
\end{defn}

It is easy to see that the P\'olya numbers $po_{K_s}$ of the Lecacheux parametric quintic fields $K_s$ are bounded above by $|Po_{K_s}|$ . We record this in the following theorem and give a quick proof.
\begin{thm}\label{NEW2}
Let $s$ be an odd integer and $K_s$ denote the Lecacheux quintic field. Then $po_{K_s}\leq|\Po(K_s)|$.
\end{thm}
The genus field (resp. genus field in the narrow sense) of $K$ is the maximal abelian extension $\Gamma_K$(resp. $\Gamma_K'$) of $K$, which is a compositum of $K$ with an absolute abelian number field and is unramified over $K$ at all places (resp. all finite places) of $K$. The genus number of $K$ is defined to be the degree $g_K=[\Gamma_K: K]$ (resp. $g_K'=[\Gamma_K': K]$). If $K$ is abelian, then Leriche showed that the genus field $\Gamma_K$(resp. $\Gamma_K'$) is P\'olya and hence
\begin{equation}\label{t1}
po_K\leq g_K \text{ \: and \: } po_K\leq g_K'.
\end{equation}

To prove \thmref{NEW2}, we need the following result due to Ishida \cite{MI} on the genus number of an absolutely abelian number field.
\begin{thm}\label{polno}\cite[Corollary, Page 52]{MI}
For an absolutely abelian number field $K$, the genus number $g_K'$ is given by
$$g_K'=\frac{\prod_p e_p}{[K : \mathbb{Q}]}$$
\end{thm}

\begin{proof} [Proof of $\thmref{NEW2}$]
    The proof follows from Eq. (\ref{t1}), Proposition \ref{CHABERT}, and Theorem \ref{polno}.
\end{proof}

\section{Computations}

\subsection{SageMath Computations} Here we present computational evidence illustrating the abundance of non-P\'olya fields within the Lecacheux family of quintic fields $K_s$, where $s$ is an odd integer. Recall that Theorem~\ref{NEW1.1} applies whenever the cube parts of both $E_1(s,1)$ and $E_2(s,1)$ are trivial. For odd integers $s$ satisfying
$-10^4\le s\le 10^4$,
our computations show that this hypothesis is satisfied for approximately $99.37\%$ of the values of $s$. More precisely, among all such values, only $63$ satisfy $C_{E_1}\neq1$ or $C_{E_2}\neq1$, where $C_{E_i}$ denotes the cube part of $E_i(s,1)$ for $i=1,2$. Moreover, within this range, we did not encounter any value of $s$ for which both $C_{E_1}\neq1$ and $C_{E_2}\neq1$ simultaneously.

In particular, for every odd integer $s$ with $-200<s<200$, Theorem~\ref{NEW1.1} applies except for $s=-125$. Furthermore, apart from the P\'olya fields corresponding to $s=\pm1$, every field in this range is non-P\'olya, and the order of the P\'olya group is given by $|\Po(K_s)|=5^{\omega(E_1E_2)-1}$. The first positive integer $s$ for which at least one of the cube parts is non-trivial is $s=575$, where $C_{E_1}=1331$. Similarly, the first negative integer for which this occurs is $s=-125$, where $C_{E_2}=1331$ (see Table~\ref{tab1}). All computations were carried out using SageMath.

{
\scriptsize
\centering
\begin{longtable}{c c c c c c | c c c c c c}
\caption{Family of P\'olya/non-P\'olya Lecacheux quintic fields}
\label{tab1} \\
\hline
$s$ & $E_1(s,1)$ & \Gape[.15cm][.15cm]{$C_{E_1}$} & $E_2(s,1)$ & $C_{E_2}$ & $\omega(E_1E_2)$ &
$s$ & $E_1(s,1)$ & $C_{E_1}$ & $E_2(s,1)$ & $C_{E_2}$ & $\omega(E_1E_2)$ \\
\hline
\endfirsthead
\multicolumn{12}{c}
{}\\
\hline
$s$ & $E_1(s,1)$ & \Gape[.15cm][.15cm]{$C_{E_1}$} & $E_2(s,1)$ & $C_{E_2}$ & $\omega(E_1E_2)$ &
$s$ & $E_1(s,1)$ & $C_{E_1}$ & $E_2(s,1)$ & $C_{E_2}$ & $\omega(E_1E_2)$ \\
\hline
\endhead
\hline
\multicolumn{12}{r}{{Continued on next page}} \\
\endfoot
\hline
\endlastfoot

% --- table data (100 rows) ---
-199 & 1584159401 & 1 & 1544755411 & 1 & 3 & 1 & 1 & 1 & 11 & 1 & 1 \\
-197 & 1521585055 & 1 & 1483357205 & 1 & 7 & 3 & 55 & 1 & 205 & 1 & 3 \\
-195 & 1460883061 & 1 & 1423807711 & 1 & 4 & 5 & 461 & 1 & 1111 & 1 & 3 \\
-193 & 1402015691 & 1 & 1366069441 & 1 & 6 & 7 & 1891 & 1 & 3641 & 1 & 4 \\
-191 & 1344945601 & 1 & 1310105291 & 1 & 5 & 9 & 5401 & 1 & 9091 & 1 & 3 \\
-189 & 1289635831 & 1 & 1255878541 & 1 & 8 & 11 & 12431 & 1 & 19141 & 1 & 3 \\
-187 & 1236049805 & 1 & 1203352855 & 1 & 4 & 13 & 24805 & 1 & 35855 & 1 & 5 \\
-185 & 1184151331 & 1 & 1152492281 & 1 & 7 & 15 & 44731 & 1 & 61681 & 1 & 3 \\
-183 & 1133904601 & 1 & 1103261251 & 1 & 3 & 17 & 74801 & 1 & 99451 & 1 & 4 \\
-181 & 1085274191 & 1 & 1055624581 & 1 & 5 & 19 & 117991 & 1 & 152381 & 1 & 2 \\
-179 & 1038225061 & 1 & 1009547471 & 1 & 5 & 21 & 177661 & 1 & 224071 & 1 & 4 \\
-177 & 992722555 & 1 & 964995505 & 1 & 6 & 23 & 257555 & 1 & 318505 & 1 & 4 \\
-175 & 948732401 & 1 & 921934651 & 1 & 6 & 25 & 361801 & 1 & 440051 & 1 & 5 \\
-173 & 906220711 & 1 & 880331261 & 1 & 3 & 27 & 494911 & 1 & 593461 & 1 & 4 \\
-171 & 865153981 & 1 & 840152071 & 1 & 5 & 29 & 661781 & 1 & 783871 & 1 & 4 \\
-169 & 825499091 & 1 & 801364201 & 1 & 5 & 31 & 867691 & 1 & 1016801 & 1 & 5 \\
-167 & 787223305 & 1 & 763935155 & 1 & 4 & 33 & 1118305 & 1 & 1298155 & 1 & 4 \\
-165 & 750294271 & 1 & 727832821 & 1 & 3 & 35 & 1419671 & 1 & 1634221 & 1 & 4 \\
-163 & 714680021 & 1 & 693025471 & 1 & 3 & 37 & 1778221 & 1 & 2031671 & 1 & 2 \\
-161 & 680348971 & 1 & 659481761 & 1 & 5 & 39 & 2200771 & 1 & 2497561 & 1 & 3 \\
-159 & 647269921 & 1 & 627170731 & 1 & 3 & 41 & 2694521 & 1 & 3039331 & 1 & 4 \\
-157 & 615412055 & 1 & 596061805 & 1 & 4 & 43 & 3267055 & 1 & 3664805 & 1 & 6 \\
-155 & 584744941 & 1 & 566124791 & 1 & 3 & 45 & 3926341 & 1 & 4382191 & 1 & 5 \\
-153 & 555238531 & 1 & 537329881 & 1 & 5 & 47 & 4680731 & 1 & 5200081 & 1 & 3 \\
-151 & 526863161 & 1 & 509647651 & 1 & 6 & 49 & 5538961 & 1 & 6127451 & 1 & 4 \\
-149 & 499589551 & 1 & 483049061 & 1 & 5 & 51 & 6510151 & 1 & 7173661 & 1 & 4 \\
-147 & 473388805 & 1 & 457505455 & 1 & 6 & 53 & 7603805 & 1 & 8348455 & 1 & 5 \\
-145 & 448232411 & 1 & 432988561 & 1 & 5 & 55 & 8829811 & 1 & 9661961 & 1 & 4 \\
-143 & 424092241 & 1 & 409470491 & 1 & 3 & 57 & 10198441 & 1 & 11124691 & 1 & 5 \\
-141 & 400940551 & 1 & 386923741 & 1 & 5 & 59 & 11720351 & 1 & 12747541 & 1 & 4 \\
-139 & 378749981 & 1 & 365321191 & 1 & 4 & 61 & 13406581 & 1 & 14541791 & 1 & 3 \\
-137 & 357493555 & 1 & 344636105 & 1 & 6 & 63 & 15268555 & 1 & 16519105 & 1 & 6 \\
-135 & 337144681 & 1 & 324842131 & 1 & 4 & 65 & 17318081 & 1 & 18691531 & 1 & 5 \\
-133 & 317677151 & 1 & 305913301 & 1 & 6 & 67 & 19567351 & 1 & 21071501 & 1 & 3 \\
-131 & 299065141 & 1 & 287824031 & 1 & 5 & 69 & 22028941 & 1 & 23671831 & 1 & 5 \\
-129 & 281283211 & 1 & 270549121 & 1 & 6 & 71 & 24715811 & 1 & 26505721 & 1 & 5 \\
-127 & 264306305 & 1 & 254063755 & 1 & 5 & 73 & 27641305 & 1 & 29586755 & 1 & 6 \\
-125 & 248109751 & 1 & 238343501 & 11 & 5 & 75 & 30819151 & 1 & 32928901 & 1 & 3 \\
-123 & 232669261 & 1 & 223364311 & 1 & 4 & 77 & 34263461 & 1 & 36546511 & 1 & 3 \\
-121 & 217960931 & 1 & 209102521 & 1 & 2 & 79 & 37988731 & 1 & 40454321 & 1 & 4 \\
-119 & 203961241 & 1 & 195534851 & 1 & 3 & 81 & 42009841 & 1 & 44667451 & 1 & 3 \\
-117 & 190647055 & 1 & 182638405 & 1 & 5 & 83 & 46342055 & 1 & 49201405 & 1 & 7 \\
-115 & 177995621 & 1 & 170390671 & 1 & 6 & 85 & 51001021 & 1 & 54072071 & 1 & 4 \\
-113 & 165984571 & 1 & 158769521 & 1 & 4 & 87 & 56002771 & 1 & 59295721 & 1 & 6 \\
-111 & 154591921 & 1 & 147753211 & 1 & 5 & 89 & 61363721 & 1 & 64889011 & 1 & 5 \\
-109 & 143796071 & 1 & 137320381 & 1 & 3 & 91 & 67100671 & 1 & 70868981 & 1 & 5 \\
-107 & 133575805 & 1 & 127450055 & 1 & 5 & 93 & 73230805 & 1 & 77253055 & 1 & 5 \\
-105 & 123910291 & 1 & 118121641 & 1 & 4 & 95 & 79771691 & 1 & 84059041 & 1 & 5 \\
-103 & 114779081 & 1 & 109314931 & 1 & 5 & 97 & 86741281 & 1 & 91305131 & 1 & 5 \\
-101 & 106162111 & 1 & 101010101 & 1 & 7 & 99 & 94157911 & 1 & 99009901 & 1 & 4 \\
-99 & 98039701 & 1 & 93187711 & 1 & 6 & 101 & 102040301 & 1 & 107192311 & 1 & 5 \\
-97 & 90392555 & 1 & 85828705 & 1 & 5 & 103 & 110407555 & 1 & 115871705 & 1 & 5 \\
-95 & 83201761 & 1 & 78914411 & 1 & 2 & 105 & 119279161 & 1 & 125067811 & 1 & 5 \\
-93 & 76448791 & 1 & 72426541 & 1 & 4 & 107 & 128674991 & 1 & 134800741 & 1 & 4 \\
-91 & 70115501 & 1 & 66347191 & 1 & 4 & 109 & 138615301 & 1 & 145090991 & 1 & 6 \\
-89 & 64184131 & 1 & 60658841 & 1 & 4 & 111 & 149120731 & 1 & 155959441 & 1 & 5 \\
-87 & 58637305 & 1 & 55344355 & 1 & 6 & 113 & 160212305 & 1 & 167427355 & 1 & 5 \\
-85 & 53458031 & 1 & 50386981 & 1 & 4 & 115 & 171911431 & 1 & 179516381 & 1 & 5 \\
-83 & 48629701 & 1 & 45770351 & 1 & 4 & 117 & 184239901 & 1 & 192248551 & 1 & 5 \\
-81 & 44136091 & 1 & 41478481 & 1 & 6 & 119 & 197219891 & 1 & 205646281 & 1 & 6 \\
-79 & 39961361 & 1 & 37495771 & 1 & 6 & 121 & 210873961 & 1 & 219732371 & 1 & 5 \\
-77 & 36090055 & 1 & 33807005 & 1 & 5 & 123 & 225225055 & 1 & 234530005 & 1 & 4 \\
-75 & 32507101 & 1 & 30397351 & 1 & 3 & 125 & 240296501 & 1 & 250062751 & 1 & 3 \\
-73 & 29197811 & 1 & 27252361 & 1 & 4 & 127 & 256112011 & 1 & 266354561 & 1 & 5 \\
-71 & 26147881 & 1 & 24357971 & 1 & 5 & 129 & 272695681 & 1 & 283429771 & 1 & 4 \\
-69 & 23343391 & 1 & 21700501 & 1 & 4 & 131 & 290071991 & 1 & 301313101 & 1 & 7 \\
-67 & 20770805 & 1 & 19266655 & 1 & 7 & 133 & 308265805 & 1 & 320029655 & 1 & 5 \\
-65 & 18416971 & 1 & 17043521 & 1 & 5 & 135 & 327302371 & 1 & 339604921 & 1 & 4 \\
-63 & 16269121 & 1 & 15018571 & 1 & 3 & 137 & 347207321 & 1 & 360064771 & 1 & 5 \\
-61 & 14314871 & 1 & 13179661 & 1 & 4 & 139 & 368006671 & 1 & 381435461 & 1 & 5 \\
-59 & 12542221 & 1 & 11515031 & 1 & 5 & 141 & 389726821 & 1 & 403743631 & 1 & 5 \\
-57 & 10939555 & 1 & 10013305 & 1 & 4 & 143 & 412394555 & 1 & 427016305 & 1 & 4 \\
-55 & 9495641 & 1 & 8663491 & 1 & 5 & 145 & 436037041 & 1 & 451280891 & 1 & 5 \\
-53 & 8199631 & 1 & 7454981 & 1 & 5 & 147 & 460681831 & 1 & 476565181 & 1 & 4 \\
-51 & 7041061 & 1 & 6377551 & 1 & 3 & 149 & 486356861 & 1 & 502897351 & 1 & 6 \\
-49 & 6009851 & 1 & 5421361 & 1 & 4 & 151 & 513090451 & 1 & 530305961 & 1 & 4 \\
-47 & 5096305 & 1 & 4576955 & 1 & 3 & 153 & 540911305 & 1 & 558819955 & 1 & 6 \\
-45 & 4291111 & 1 & 3835261 & 1 & 3 & 155 & 569848511 & 1 & 588468661 & 1 & 6 \\
-43 & 3585341 & 1 & 3187591 & 1 & 5 & 157 & 599931541 & 1 & 619281791 & 1 & 4 \\
-41 & 2970451 & 1 & 2625641 & 1 & 4 & 159 & 631190251 & 1 & 651289441 & 1 & 4 \\
-39 & 2438281 & 1 & 2141491 & 1 & 3 & 161 & 663654881 & 1 & 684522091 & 1 & 6 \\
-37 & 1981055 & 1 & 1727605 & 1 & 6 & 163 & 697356055 & 1 & 719010605 & 1 & 5 \\
-35 & 1591381 & 1 & 1376831 & 1 & 4 & 165 & 732324781 & 1 & 754786231 & 1 & 5 \\
-33 & 1262251 & 1 & 1082401 & 1 & 4 & 167 & 768592451 & 1 & 791880601 & 1 & 5 \\
-31 & 987041 & 1 & 837931 & 1 & 4 & 169 & 806190841 & 1 & 830325731 & 1 & 4 \\
-29 & 759511 & 1 & 637421 & 1 & 3 & 171 & 845152111 & 1 & 870154021 & 1 & 3 \\
-27 & 573805 & 1 & 475255 & 1 & 4 & 173 & 885508805 & 1 & 911398255 & 1 & 5 \\
-25 & 424451 & 1 & 346201 & 1 & 2 & 175 & 927293851 & 1 & 954091601 & 1 & 3 \\
-23 & 306361 & 1 & 245411 & 1 & 3 & 177 & 970540561 & 1 & 998267611 & 1 & 8 \\
-21 & 214831 & 1 & 168421 & 1 & 4 & 179 & 1015282631 & 1 & 1043960221 & 1 & 5 \\
-19 & 145541 & 1 & 111151 & 1 & 5 & 181 & 1061554141 & 1 & 1091203751 & 1 & 6 \\
-17 & 94555 & 1 & 69905 & 1 & 5 & 183 & 1109389555 & 1 & 1140032905 & 1 & 8 \\
-15 & 58321 & 1 & 41371 & 1 & 3 & 185 & 1158823721 & 1 & 1190482771 & 1 & 6 \\
-13 & 33671 & 1 & 22621 & 1 & 3 & 187 & 1209891871 & 1 & 1242588821 & 1 & 3 \\
-11 & 17821 & 1 & 11111 & 1 & 4 & 189 & 1262629621 & 1 & 1296386911 & 1 & 6 \\
-9 & 8371 & 1 & 4681 & 1 & 4 & 191 & 1317072971 & 1 & 1351913281 & 1 & 4 \\
-7 & 3305 & 1 & 1555 & 1 & 3 & 193 & 1373258305 & 1 & 1409204555 & 1 & 6 \\
-5 & 991 & 1 & 341 & 1 & 3 & 195 & 1431222391 & 1 & 1468297741 & 1 & 3 \\
-3 & 181 & 1 & 31 & 1 & 2 & 197 & 1491002381 & 1 & 1529230231 & 1 & 5 \\
-1 & 11 & 1 & 1 & 1 & 1 & 199 & 1552635811 & 1 & 1592039801 & 1 & 3 \\
\end{longtable}

}

\subsection{Wolfram Mathematica Script} 
The following Wolfram Mathematica script, used in \propref{prop1}, computes elements 
\(\delta \in \mathbb{Q}(\zeta_5)\) satisfying
\[
N_{\mathbb{Q}(\zeta_5)/\mathbb{Q}}(\delta) = E_1(s,1) := s^4 - 2s^3 + 4s^2 - 3s + 1.
\]

\begin{verbatim}(*Define the input polynomial of degree 4*)
Polynomial = 1 - 3*s + 4*s^2 - 2*s^3 + s^4;

(*Define the element \[Alpha]*)
\[Alpha] = s + a0 + a1*\[Zeta] + a2*\[Zeta]^2 + a3*\[Zeta]^3;

(*Find all conjugates of \[Alpha]*)
Conjugates = 
  Table[s + a0 + a1*\[Zeta]^k + a2*\[Zeta]^(2*k) + 
    a3*\[Zeta]^(3*k), {k, 1, 4}];

(*Simplify each conjugate to reduce powers modulo 5*)
ConjugatesSimplified = 
  Expand[Conjugates /. 
    Power[\[Zeta], m_] :> Power[\[Zeta], Mod[m, 5]]];

(*Define the primitive 5th root of unity \[Zeta]*)
\[Zeta] = Exp[2 Pi I/5];

(*Compute the norm as the product of all conjugates*)
NormAlpha = Times @@ ConjugatesSimplified;

(*Fully simplify the norm as a polynomial in s*)
NormSimplified = Expand[FullSimplify[NormAlpha]];

(*Extract coefficients of the simplified norm polynomial*)
s0 = Coefficient[NormSimplified, s, 0] // FullSimplify;
s1 = Coefficient[NormSimplified, s, 1] // FullSimplify;
s2 = Coefficient[NormSimplified, s, 2] // FullSimplify;
s3 = Coefficient[NormSimplified, s, 3] // FullSimplify;

(*Solve the system of equations for a0,a1,a2,a3*)
Solutions = 
  Solve[s0 == Coefficient[Polynomial, s, 0] && 
    s1 == Coefficient[Polynomial, s, 1] && 
    s2 == Coefficient[Polynomial, s, 2] && 
    s3 == Coefficient[Polynomial, s, 3], {a0, a1, a2, a3}, Integers];

Print["All possible choices for \[Alpha]:"];

Do[Print["Solution ", i, ":   ", "\[Alpha] = s + ", 
   a0 /. Solutions[[i]], " + ", a1 /. Solutions[[i]], "\[Zeta] + ", 
   a2 /. Solutions[[i]], 
   "\!\(\*SuperscriptBox[\(\[Zeta]\), \(2\)]\) + ", 
   a3 /. Solutions[[i]], 
   "\!\(\*SuperscriptBox[\(\[Zeta]\), \(3\)]\)"], {i, 1, 4}];

\end{verbatim}

\section*{Acknowledgements}
We are grateful to Prof. Claude Levesque for suggesting this problem. The first author would like to thank the Roman Number Theory Association and Prof. Francesco Pappalardi for inviting to Leuca 2022, the sixth mini symposium of the Roman Number Theory Association, Italy, where this work was initiated.

\end{document}